\documentclass{amsart}

\usepackage{graphicx,algorithm,algpseudocode,xcolor,stmaryrd,amsaddr,tikz,hyperref}

\algnewcommand\algorithmicinput{\textbf{Input:}}
\algnewcommand\Input{\item[\algorithmicinput]}
\algnewcommand\algorithmicoutput{\textbf{Output:}}
\algnewcommand\Output{\item[\algorithmicoutput]}
\algnewcommand{\LineIf}[2]{\State \algorithmicif\, #1 \,\algorithmicthen\, #2 \,\algorithmicend\ \algorithmicif}
\algnewcommand{\LineForAll}[2]{\State \algorithmicforall\, #1 \,\algorithmicdo\, #2 \,\algorithmicend\ \algorithmicfor}
\algnewcommand{\Accept}{\textbf{accept}}
\algnewcommand{\Reject}{\textbf{reject}}

\newcommand{\Model}{\mathrm{Model}}
\newcommand{\Modelbf}{\mathbf{Model}}

\newcommand{\dom}{\mathrm{dom}}
\newcommand{\image}{\mathrm{image}}

\newcommand{\eq}{\ensuremath\mathrel{=}}
\newcommand{\NL}{\ensuremath\mathsf{NL}}
\newcommand{\NP}{\ensuremath\mathsf{NP}}
\newcommand{\co}{\ensuremath\mathsf{co}}
\newcommand{\PSPACE}{\ensuremath\mathsf{PSPACE}}

\newcommand{\EXPTIME}{\ensuremath\mathsf{EXPTIME}}

\newcommand{\AC}{\ensuremath\mathsf{AC}}
\newcommand{\LOGSPACE}{\ensuremath\mathsf{L}}
\newcommand{\Poly}{\ensuremath\mathsf{P}}
\newcommand{\gR}{\mathcal{R}}

\newtheorem{theorem}{Theorem}[section]
\newtheorem{lemma}[theorem]{Lemma}
\newtheorem{corollary}[theorem]{Corollary}
\newtheorem{proposition}[theorem]{Proposition}

\theoremstyle{remark}

\numberwithin{equation}{section}
\begin{document}

\title{On the Complexity of Properties of Partial Bijection Semigroups}
\date{\today}
\author{Trevor Jack}
\address{Illinois Wesleyan University \\
Department of Mathematics \\
1312 Park St.\\
Bloomington, IL 61701}
\email{trevjack@gmail.edu}a

\thanks{This work was partially supported by the Funda\c{c}\~{a}o para a Ci\^{e}ncia e a Tecnologia (Portuguese Foundation for Science and Technology) through the projects UIDB/00297/2020 (Centro de Matemática e Aplica\c{c}\~{o}es) and PTDC/MAT-PUR/31174/2017.}

\begin{abstract}
We examine the computational complexity of problems in which we are given generators for a partial bijection semigroup and asked to check properties of the generated semigroup. We prove that the following problems are in $\AC^0$: (1) enumerating left and right identities and (2) checking if the semigroup is completely regular. We also describe a nondeterministic logspace algorithm for checking if an inverse semigroup given by generators satisfies a fixed semigroup identity that may involve a unary inverse operation. We conclude with an alternative proof that checking membership of a given idempotent in a partial bijection semigroup is a $\PSPACE$-complete problem. The proof reduces from the well-known $\PSPACE$-complete Rectangle Tiling Problem, thereby illustrating a connection between Wang tilings and partial bijection semigroups.
\end{abstract}
\maketitle
\section{Introduction}

Given generators for a finite group of permutations, Sims' stabilizer chains can determine many properties of the generated group in polynomial time, $\Poly$, such as checking membership and calculating size. In contrast, a well-known result by Kozen \cite{KO:LBN} is that checking membership in a transformation semigroup given by generators is among the hardest problems solvable in polynomial space: that is, it is a $\PSPACE$-complete problem. Furthermore, the known algorithms for computing various properties of transformation semigroups, such as size, often rely upon an enumeration of the $\gR$-classes of the semigroup, which already requires exponential time \cite{EA:CFS,MI:GAP}.

There are canonical embeddings of groups into partial bijection semigroups and from partial bijection semigroups into transformation semigroups. So, in some sense, the complexity of partial bijection semigroups can be thought of as lying between the complexity of groups and the complexity of transformation semigroups. The authors in \cite{FJ:CP} analyzed the complexity of various problems in which we are given generators for a transformation semigroup and asked to check properties of the generated semigroup. This paper analyzes the complexity of corresponding problems in which the generators are partial bijections. In particular, the following problems are known to be in $\NL$ for transformation semigroups and this paper proves they are in $\AC^0$ for partial bijection semigroups:

\begin{itemize}
\item checking if the semigroup is a band;
\item checking if the semigroup is completely regular; and
\item checking if the semigroup is Clifford.
\end{itemize}

We improve upon an algorithm from \cite{FJ:CP} that enumerates left and right identities in polynomial time ($\Poly$). We now prove that this problem is in $\AC^0$. The authors in \cite{FJ:CP} also describe a nondeterministic algorithm that runs in logarithmic space ($\NL$) for checking if a transformation semigroup given by generators satisfies a fixed semigroup identity. This problem is the dual of the well-known identity checking problem for which the semigroup is fixed and the identity is given. See \cite{AVG:CIC} for background and complexity results on identity checking: in particular, examples of semigroups for which the identity checking problem is $\co\NP$-complete. This paper extends the algorithm from \cite{FJ:CP} to inverse semigroups. We allow the fixed semigroup identity to involve a unary inverse operation and describe an $\NL$ algorithm for determining if an inverse semigroup given by generators satisfies the identity.

We finally  consider the problem of checking membership in inverse semigroups, which can be thought of as partial bijection semigroups that contain unique inverses for each of their elements \cite[Thm 5.1.7]{HO:FST}. \cite[Thm 4.10]{TJ:CI} proves that the problem is $\PSPACE$-complete. This paper gives an alternative proof by reducing from the $\PSPACE$-complete Rectangle Tiling Problem and the argument is adaptable to any decision problem involving Wang tiling.

\section{Preliminaries} \label{NotationSection}

For $n \in \mathbb{N}$, let $[n] := \{1,\dots,n\}$. The {\bf full transformation semigroup} over $[n]$, denoted $T_n$, is the set of all mappings $f \colon [n] \to [n]$, together with function composition. For elements $a_1,\dots,a_k \in T_n$, let $\langle a_1,\dots,a_k\rangle$ be the subsemigroup generated by $a_1,\dots,a_k$. Subsemigroups of the full transformation semigroup are often also referred to as {\bf transformation semigroups}.

The {\bf full partial bijection semigroup} over $[n]$, denoted $I_n$, is the set of all partial bijective mappings $f \colon [n] \to [n]$, together with function composition. For $a,b \in I_n$, we define
\[ \dom(a) := \{q \in [n]: \exists p \in [n](qa = p)\} \]
\[\image(a) := \{q \in [n] : \exists p \in [n] (pa = q)\}\]

Subsemigroups of the full partial bijection semigroup are often also referred to as {\bf partial bijection semigroups}. This paper will be investigating the computational complexity of decision problems in which we are given a set of partial bijections and asked to check some property of the generated partial bijection semigroup. Our analysis will reference complexity classes from the following hierarchy:

$$ \AC^0 \subseteq \LOGSPACE \subseteq \NL \subseteq \Poly \subseteq \PSPACE \subseteq \EXPTIME.$$

$\AC^0$ is the class of sets decidable by uniform unbounded fan-in Boolean circuits of polynomial size and constant depth. Equivalently, $\AC^0$ is the class of first-order definable properties \cite{IM:DC}. Hence, it includes all decision problems in which we are given generators $a_1,\dots,a_k \in I_n$ and asked to check a property that can be characterized by a first-order formula quantified over the points $[n]$ and the generators $\{a_1,\dots,a_k\}$. Note that we are not allowed to quantify over all generated elements of the semigroup. $\LOGSPACE$ ($\NL$) is the class of sets decidable by a deterministic (nondeterministic) Turing machine using at most logarithmic space. $\Poly$ ($\PSPACE$) consist of sets that are decidable by a deterministic Turing machine in polynomial time (space). We refer the reader to \cite{PA:CC} for further background on computational complexity and to \cite{CP:AT} and \cite{HO:FST} for further background on semigroup theory.

\section{Extending Results For $T_n$ to $I_n$}
There is a natural representation of a partial bijection $a \in I_n$ as a transformation $a' \in T_{n+1}$, where $xa' = xa$ for $x \in \dom(a)$ and $xa' = n+1$ for $x \not \in \dom(a)$. Thus, partial bijection semigroup problems are at most as difficult as their corresponding transformation semigroup problems. Corollary~\ref{cor:red} follows from applying this fact to results from \cite{FJ:CP}, where relevant definitions for these results are discussed. In particular, we recall the following definitions.

A semigroup element $0 \in S$ is a \emph{left (right) zero} if $0a = 0$ ($a0 = 0$) for each $a \in S$. If a semigroup has a left zero and a right zero, then they are equal and we call this the zero element of the semigroup. A semigroup with a zero element is \emph{$n$-nilpotent} if the composition of any $n$ elements of the semigroup yields the zero element. A semigroup is \emph{nilpotent} if it is $n$-nilpotent for some $n \in \mathbb{N}$. Two elements $a,b \in S$ are \emph{$\gR$-related} if $aS \cup \{a\} = bS \cup \{b\}$. A semigroup is \emph{$\gR$-trivial} if no two distinct elements in the semigroup are $\gR$-related. A semigroup $S$ is \emph{regular} if for every element $a \in S$ there exists $b \in S$ such that $aba = a$.

\begin{corollary} \label{cor:red}
Checking if a partial bijection semigroup given by generators:
\begin{enumerate}
\item is commutative is in $\AC^0$ by \cite[Thm 3.2]{FJ:CP},
\item is a semilattice is in $\AC^0$ by  \cite[Thm 3.3]{FJ:CP},
\item is a group is in $\AC^0$ by  \cite[Thm 3.5]{FJ:CP},
\item has left zeroes, right zeroes, or a zero is in $\NL$ by \cite[Thm 4.6]{FJ:CP}.
\end{enumerate}
\end{corollary}

Several problems discussed in \cite{FJ:CP} have tighter upper complexity bounds for partial bijection semigroups than for transformation semigroups. For example, we have $\NL$ algorithms for checking if a transformation semigroup has commuting idempotents and whether the product of any two idempotents is idempotent. But these properties are always true for partial bijection semigroups. Also, we have an $\NL$ algorithm for checking if a transformation semigroup is a band, but we can do better for partial bijection semigroups. Idempotents in partial bijection semigroups are maps that fix their domain, so idempotents of a partial bijection semigroup commute. Thus, a partial bijection semigroup is a band iff it is a semilattice, which can be checked in $\AC^0$ by \cite[Thm 3.3]{FJ:CP}.

We now consider the problem of determining if a partial bijection semigroup is completely regular. There are several equivalent characterizations of completely regular semigroups \cite[Prop 4.1.1]{HO:FST}. We say a semigroup is {\bf completely regular} if each of its elements generates a subgroup of the semigroup. Determining if a transformation semigroup given by generators is completely regular is in $\NL$ \cite[Thm 5.6]{FJ:CP}, but we can give a stronger result for partial bijection semigroups.

\medskip
{\bf Completely Regular}
\begin{itemize}
\item Input: $a_1,\dots,a_k \in I_n$.
\item Problem: Is $\langle a_1,\dots,a_k\rangle$ completely regular?
\end{itemize}

\begin{theorem} \label{thm:ComReg}
Completely Regular is in $\AC^0$.
\end{theorem}
\begin{proof}
Let $S = \langle a_1,\dots,a_k \rangle \leq I_n$. We claim that $S$ is completely regular iff the following first-order formula holds:
\[\forall i,j \in [k] : \dom(a_ia_j) = \dom(a_i) \cap \dom(a_j).\]

Assume $S$ is completely regular and pick any $i,j \in [k]$. We first prove that $\dom(a_ia_j) \subseteq \dom(a_i) \cap \dom(a_j)$. Pick any $x \in \dom(a_ia_j)$. Certainly $x \in \dom(a_i)$. A transformation generates a subgroup iff it permutes its own image. For partial bijections $a_i$ and $a_j$, this means $\dom(a_ia_j) = \image(a_ia_j)$. Thus, $x \in \image(a_j)$ which in turn forces $x \in \dom(a_j)$. To prove $\dom(a_i) \cap \dom(a_j) \subseteq \dom(a_ia_j)$, note that $\dom(a_i) = (\dom(a_i) \setminus \dom(a_j)) \cup (\dom(a_i) \cap \dom(a_j))$. Pick any $x \in \dom(a_i) \setminus \dom(a_j)$. Then $x \not \in \image(a_j)$, $x \not \in \image(a_ia_j) = \dom(a_ia_j)$, and thus $xa_i \not \in \dom(a_j)$. This proves that, in addition to being a bijection on its domain, $a_i$ is also a bijection on $\dom(a_i) \setminus \dom(a_j)$. Consequently, $a_i$ is a bijection on $\dom(a_i) \cap \dom(a_j)$ so that for any $x \in \dom(a_i) \cap \dom(a_j)$, we know $xa_i \in \dom(a_i) \cap \dom(a_j)$. Since $xa_i \in \dom(a_j)$, then $x \in \dom(a_ia_j)$.

Conversely, assume the first-order formula holds. We claim that $\dom(st) = \dom(s) \cap \dom(t)$ for every $s,t \in S$. Because this would yield $\dom(ss) = \dom(s) \cap \dom(s) = \dom(s)$, our claim would prove that every element permutes its own image. Pick any $s = s_1 \cdots s_\ell$ and $t = t_1 \cdots t_m$ with $s_1,\dots,s_\ell,t_1,\dots t_m \in \{a_1,\dots,a_k\}$. Pick any $x \in \dom(st)$. Certainly, $x \in \dom(s)$. For the sake of contradiction, suppose $x \not \in \dom(t)$. Then there is either an $s_i$ or $t_i$ satisfying one of the following two consequences: (1) $xs_1 \cdots s_{i-1} \not \in \dom(t_1)$ and $xs_1 \cdots s_i \in \dom(t_1)$ or (2) $xs t_1 \cdots t_{i-1} \not \in \dom(t_{i+1})$ and $xs t_1 \cdots t_i \in \dom(t_{i+1})$. Then either $\dom(s_it_1) \neq \dom(s_i) \cap \dom(t_1)$ or $\dom(t_it_{i+1}) \neq \dom(t_i) \cap \dom(t_{i+1})$. Both contradict the first-order formula, so $\dom(st) \subseteq \dom(s) \cap \dom(t)$.

Pick any $x \in \dom(s) \cap \dom(t)$. For the sake of contradiction, suppose that $xs \not \in \dom(t)$. Then there exists either $s_i$ or $t_i$ satisfying one of the following consequences: (1) $xs_1 \cdots s_{i-1} \in \dom(t_1)$ and $xs_1 \cdots s_i \not \in \dom(t_1)$ or (2) $xst_1 \cdots t_{i-1} \in \dom(t_{i+1})$ and $xst_1 \cdots t_i \not \in \dom(t_{i+1})$. Again, both consequences contradict the first-order formula, so $\dom(st) = \dom(s) \cap \dom(t)$.
\end{proof}

Several corollaries to \cite[Thm 5.6]{FJ:CP} can be analogously refined for partial bijection semigroups. Given generators for a transformation semigroup, there are $\NL$ algorithms to test for each of the following properties of the generated semigroup: (a) whether the semigroup is Clifford and (b) whether the semigroup, if commutative, is also regular. Recall that  idempotents of a partial bijection semigroup commute, so every completely regular partial bijection semigroup is also a Clifford semigroup \cite[Def 4.2.1]{HO:FST}. Furthermore, a commutative semigroup is regular iff it is completely regular. Thus, Theorem~\ref{thm:ComReg} yields the following corollaries.

\begin{corollary} \label{cor:Cliff}
Deciding whether a partial bijection semigroup given by generators is a Clifford semigroup is in $\AC^0$.
\end{corollary}

\begin{corollary} \label{cor:Reg}
Deciding whether a commutative partial bijection semigroup given by generators is a regular semigroup is in $\AC^0$.
\end{corollary}

We now consider the following problem.

\medskip
{\bf Enumerate Identities}
\begin{itemize}
\item Input: $a_1,\dots,a_k \in T_n$.
\item Output: The left and right identities of $\langle a_1, \dots, a_k\rangle$.
\end{itemize}

Recall that an element $\ell$ (resp. $r$) of a semigroup $S$ is a {\bf left} (resp. {\bf right}) {\bf identity} if $\ell s = s$ (resp. $sr = s$) for all $s \in S$. It has been previously shown that left and right identities in transformation semigroups are idempotent powers of generators \cite[Lem 6.1 and Lem 6.3]{FJ:CP} and that they can be enumerated in polynomial time \cite[Thm 6.2 and Thm 6.4]{FJ:CP}. We now prove that Enumerate Identities is in $\AC^0$.

\begin{lemma} \label{lem:leftid}
Let $S := \langle a_1, \dots, a_k \rangle \leq T_n$. Then the idempotent power of $a_i$ is a left identity of $S$ iff
\[ \forall j \in [k] \, \forall x,y \in [n]: (x a_i = y a_i \implies x a_j = y a_j) \wedge (x a_i^2 = y a_i^2 \implies x a_i = y a_i) \]
\end{lemma}
\begin{proof}
Assume $\ell \in S$ is a left identity. By \cite[Lemma 6.1]{FJ:CP}, $\ell$ is the idempotent power of some generator: $\ell = a_i^\omega$. If $x a_i = y a_i$, then $x a_i^\omega a_j = y a_i^\omega a_j$ and thus $x a_j = y a_j$. If $x a_i^2 = y a_i^2$, then $x a_i^{\omega+1} = y a_i^{\omega+1}$ and thus $x a_i = y a_i$.

Assume $a_i$ satisfies the first-order formula and let $\ell = a_i^\omega$ be its idempotent power. Pick any $j \in [k]$ and any $x \in [n]$. Starting with $x\ell^2 = x\ell$ and repeatedly applying the second clause of the first-order formula, we obtain $x\ell a_i = x a_i$. Then the first clause yields $x\ell a_j = x a_j$. Thus, $\ell$ is a left identity.
\end{proof}

\begin{lemma} \label{lem:rightid}
Let $S := \langle a_1, \dots, a_k \rangle \leq T_n$. Then the idempotent power of $a_i$ is a right identity of $S$ iff
\[ \forall j,\ell \in [k] \, \forall x,y \in [n]: x a_j a_i = y a_\ell a_i \implies x a_j = y a_\ell \]
\end{lemma}
\begin{proof}
Assume $r \in S$ is a right identity. By \cite[Lemma 6.3]{FJ:CP}, $r$ is the idempotent power of some generator: $r = a_i^\omega$. If $xa_j a_i = y a_\ell a_i$, then $xa_j a_i^\omega = y a_\ell a_i^\omega$. Since $a_i^\omega$ is a right identity, then $xa_j = y a_\ell$.

Assume $a_i$ satisfies the first-order formula and let $r = a_i^\omega$ be its idempotent power. Pick any $j \in [k]$ and any $x \in [n]$. Starting with $xa_jr^2 = xa_jr$, we can use the formula to remove copies of $a_i$ until we are left with $xa_jr = xa_j$. Thus, $r$ is a right identity.
\end{proof}

\begin{theorem}
Enumerate Identities is in $\AC^0$.
\end{theorem}
\begin{proof}
For each generator $a_i$, we can define an $\AC^0$ circuit to check the formulas in Lemma~\ref{lem:leftid} and Lemma~\ref{lem:rightid}. If this circuit confirms that $a_i^\omega$ is a left or right identity, we would then like to identify the images $xa_i^\omega$ for each $x \in [n]$. Fortunately, the conditions in each of the lemmas offer us a way around direct computation.

Consider any $x,y \in [n]$. Since $xa_i = xa_i^{\omega+1}$, then $xa_i = ya_i^2$ implies $xa_i^{\omega+1} = ya_i^2$. By the conditions in each Lemma, this implies that $xa_i^\omega = ya_i$. So, the value of $xa_i^\omega$ equals $ya_i$ where $y$ satisfies $xa_i = ya_i^2$. Then, for each $x \in [n]$, each generator $a_i$ that satisfies the formulas in either Lemm~\ref{lem:leftid} or Lemma~\ref{lem:rightid}, and each $y \in [n]$, we define an $\AC^0$ to check if $xa_i = ya_i^2$. If the circuit accepts, then $xa_i^\omega = ya_i$

\end{proof}
Note that a semigroup has a two-sided identity iff it has a left identity and a right identity, so enumerating the two-sided identity is also in $\AC^0$.

We now consider the $\NL$ algorithm described in \cite[Thm 5.1]{FJ:CP} for checking if a transformation semigroup given by generators satisfies a given semigroup identity, We now generalize that algorithm to inverse partial bijection semigroups and semigroup identities that may involve a unary inverse operation. Let $a_1,\dots,a_k$ be partial bijective maps, each defined on subsets of $[n]$, and let $S= \langle a_1,\dots,a_k,a_1^{-1},\dots,a_k^{-1}\rangle$. Let $X^*$ be the free inverse monoid over the variables $X = \{x_1,\dots,x_m,x_1^{-1},\dots,x_m^{-1}\}$. A map $h:X^* \to S$ is a {\bf homomorphism} if $h(xy) = h(x)h(y)$ and $h(x^{-1}) = h(x)^{-1}$ for each $x,y \in X^*$. Let $u$ and $v$ be two elements of $X^*$. We say that an inverse semigroup $S$ {\bf models} $u \eq v$ if $h(u) = h(v)$ holds for each homomorphism $h:X^* \to S$. For a fixed identity $u \eq v$, define the following problem:

\medskip
{$\Modelbf(u \eq v)$}
\begin{itemize}
\item Input: $a_1,\dots,a_k \in I_n$
\item Problem: Does $\langle a_1,\dots,a_k,a_1^{-1},\dots,a_k^{-1} \rangle$ model $u \eq v$?
\end{itemize} 

We will show that this class of problems belongs to $\NL$ by showing that a broader class of problems also belongs to $\NL$. We say that an inverse semigroup $S$ {\bf models} $x_1 = x_1^2,\dots,x_e = x_e^2 \implies u = v$, where $e \leq m$, if for all homomorphisms $h \colon X^* \to S$ with $h(x_1), \dots, h(x_e)$ idempotent, we have $h(u) = h(v)$.

\medskip
{$\Modelbf(x_1 \eq x_1^2,\dots,x_e \eq x_e^2 \implies u \eq v)$}
\begin{itemize}
\item Input: $a_1,\dots,a_k \in I_n$
\item Problem: Does $\langle a_1,\dots,a_k \rangle$ model $x_1 \eq x_1^2,\dots,x_e \eq x_e^2 \implies u \eq v$?
\end{itemize}

\begin{theorem}
  Let $X = \{x_1, \dots, x_m,x_1^{-1},\dots,x_m^{-1}\}$ be a nonempty finite set of variables and let $u, v \in X^*$. Then, $\Model(x_1 \eq x_1^2,\dots,x_e \eq x_e^2 \implies u \eq v)$ belongs to $\NL$.
  \label{thm:model}
\end{theorem}
\begin{proof}
  Let $u \eq x_{i_1}^{f_1} \cdots x_{i_\ell}^{f_\ell}$ and $v \eq x_{j_1}^{g_1} \cdots x_{j_r}^{g_r}$ with $i_1,\dots,i_\ell,j_1,\dots,j_r \in [m]$ and $f_1,\dots,f_\ell,g_1,\dots,g_r \in \{-1,1\}$. We describe an $\NL$ algorithm to test whether an inverse semigroup $S = \langle a_1, \dots, a_k,a_1^{-1},\dots,a_k^{-1} \rangle$ does \emph{not} model
  \begin{align*}
    x_1 \eq x_1^2,\dots,x_e \eq x_e^2 \implies u \eq v.
  \end{align*}
  
  Since $\NL$ is closed under complementation \cite[Thm 8.27]{SIP:ITC}, this implies that the decision problem $\Model(x_1 \eq x_1^2,\dots,x_e \eq x_e^2 \implies u \eq v)$ belongs to $\NL$.
  For each $i \in [m]$, we let $P_1(i) = \{p \in [\ell]:i_p = i\}$ and $P_2(i) = \{p \in [r]:j_p = i\}$.
  The algorithm is depicted in Algorithm~\ref{alg:models}.
  Since $\ell + r$ is a constant, the algorithm only requires logarithmic
  space.

  \begin{algorithm}
  \caption{$\co\NL$ algorithm for $\Model(u \eq v)$}
    \label{alg:models}
    \begin{algorithmic}[1]
      \Input{$a_1,\dotsc,a_k \in I_n$}
      \Output{Does $\langle a_1,\dots,a_k,a_1^{-1},\dots,a_k^{-1} \rangle$ \emph{not} model $u \eq v$?}
      \State guess integers $p_1, \dots, p_{\ell+2}, q_1, \dots, q_{r+2} \in [n+1]$
      \LineIf{$p_1 \ne q_1$ or $p_{\ell+1} = q_{r+1}$}{\Reject}
      \ForAll{$i \in [m]$}
      \LineForAll{$j \in [\ell]$}{$p_j' := p_j$,\, $p_j'' := p_{j+1}$}
      \LineForAll{$j \in [r]$}{$q_j' := q_j$,\, $q_j'' := q_{j+1}$}
        \Repeat
          \State guess $c \in [k]$
          \ForAll{$j \in P_1(i)$}
            \LineIf{$f_j = 1$}{$p_j' := p_j'a_c$,\, $p_j'' := p_j''a_c$}
            \LineIf{$f_j = -1$}{$p_{j+1}' := p_{j+1}'a_c,\, p_{j+1}'' := p_{j+1}'' a_c$}
          \EndFor
          \ForAll{$j \in P_2(i)$}
            \LineIf{$g_j = 1$}{$q_j' := q_j'a_c$,\, $q_j'' := q_j''a_c$}
            \LineIf{$g_j = -1$}{$q_{j+1}' := q_{j+1}'a_c,\, q_{j+1}'' := q_{j+1}''a_c$}
          \EndFor
         \Until{[$\forall j \in P_1(i) \colon p_j' = p_{j+1}$ if $f_j = 1$ and $p_{j+1}' = p_j$ if $f_j = -1]$ and \\ \hspace{1.35cm} $[\forall j \in P_2(i) \colon q_j' = q_{j+1}$ if $g_j = 1$ and $q_{j+1}' = q_j$ if $g_j = -1]$ and \\ \hspace{1.35cm} $[i \in [e] \implies (\forall j \in P_1(i) \colon p_j'' = p_{j+1}$ if $f_j = 1$ and \\ \hspace{1.35cm} $p_{j+1}'' = p_j$ if $f_j = -1)$ and $(\forall j \in P_2(i) \colon q_j'' = q_{j+1}$ if $g_j = 1$\\ \hspace{1.35cm}  and $q_{j+1}'' = q_j$ if $g_j = -1)]$}
      \EndFor
      \State\Accept
    \end{algorithmic}
  \end{algorithm}

  The process corresponds to nondeterministically guessing generators to produce elements from $S$ to substitute for the variables in $X$ such that the left-hand side and the right-hand
  side of the equation map the point $p_1 = q_1 \in [n]$ to distinct points $p_{\ell+1},q_{r+1} \in [n+1]$ (i.e. a point $x$ such that $xu \neq xv$). The extra point $n+1$ represents the image of a generator applied to a point outside of the generator's domain and this extra point is fixed by all generators This is the natural embedding of partial bijections into a transformation semigroup. The points $p_2,\dots,p_\ell,q_2,\dots,q_r \in [n+1]$ represent evaluations after applying the generators comprising $u$ and $v$. We will need two more points, $p_{\ell+1}$ and $q_{r+1}$ for technical details discussed in the following correctness proof.
  
  First, suppose that the input $S = \langle a_1,\dots,a_k\rangle$ does not model $x_1 \eq x_1^2,\dots,x_e \eq x_e^2 \implies u \eq v$. This means that there are elements $s_1,\dots,s_m \in S$ such that $s_{i_1}^{f_1}\cdots s_{i_\ell}^{f_\ell} \neq s_{j_1}^{g_1}\cdots s_{j_r}^{g_r}$ and $s_1,\dots,s_e$ are idempotent. Pick a $p_1 \in [n]$ such that $p_1 s_{i_1}^{f_1} \cdots s_{i_\ell}^{f_\ell} \neq p_1 s_{j_1}^{g_1} \cdots s_{j_r}^{g_r}$. Let $q_1 := p_1$. For each $\alpha \in [\ell]$, let $p_\alpha = p_1 s_{i_1}^{f_1} \cdots s_{i_{\alpha-1}}^{f_{\alpha-1}}$. For each $\alpha \in [r]$, let $q_\alpha = q_1 s_{j_1}^{g_1} \cdots s_{j_{\alpha-1}}^{g_{\alpha-1}}$.

  To verify that the algorithm will accept the input, consider any $s_i \in \{s_1,\dots,s_m\}$. Let $s_i = a_{c_1} \cdots a_{c_d}$ with $c_1, \dots, c_d \in [k]$. Lines 6--17 will successively guess the generators and transform the points $p_j'$ and $p_{j+1}'$ for each $j \in P_1(i)$; likewise for $q_j'$ and $q_{j+1}'$ for each $j \in P_2(i)$. When this loop completes, the algorithm will ensure the following. For each $j \in P_1(i)$, $f_i = 1$ implies $p'_j = p_js_i = p_{j+1}$ and $f_i = -1$ implies $p'_{j+1} = p_{j+1}s_i$. Note that, in the latter case, $p_js_i^{-1} = p_{j+1}$, so $p_j = p_{j+1}s_i$ and thus $p'_{j+1} = p_j$. The algorithm works likewise for the points $q_j'$ and $q_{j+1}'$ for each $j \in P_2(i)$. Finally, Lines 18-20 are satisfied since $s_i$ and $s_i^{-1}$ are idempotent. If $f_j = 1$, $p''_j = p_{j+1}s_i = (p_js_i)s_i = p_js_i = p_{j+1}$. If $f_j = -1$, $p''_j = p_{j+1}s^{-1}_i = (p_js^{-1}_i)s^{-1}_i = p_js^{-1}_i = p_{j+1}$. We use the points $p_{\ell+2}$ and $q_{r+2}$ to define $p''_{\ell+1}$ and $q''_{r+1}$ in the case that $f_\ell=-1$ or $g_r=-1$, respectively.
  
  We now prove that if the algorithm accepts, then $S = \langle a_1,\dots,a_k\rangle$ does not model $x_1 \eq x_1^2,\dots,x_e \eq x_e^2 \implies u \eq v$.
  Let $p_1,\dots,p_{\ell+1},q_1,\dots,q_{r+1}$ be the guessed integers in Line 1. For each $i \in [m]$, let $s_i = a_{c_1}\cdots a_{c_g}$ be the sequence of guessed generators in Line 7.
  Then for each $j \in P_1(i)$: $p_j s_i = p_{j+1}$ if $f_i=1$ and $p_j s_i^{-1} = p_{j+1}$ if $f_i = -1$. Likewise, for each $j \in P_2(i)$, $q_j s_i = q_{j+1}$ if $g_i=1$ and $q_j s_i^{-1} = q_{j+1}$ if $g_i=-1$. Let $s_i^\omega$ be the idempotent power of $s_i$. Then for each $j \in P_1(i)$ with $j \le e$, we have $p_{j+1}s_i^\omega = p_{j+1} s_i = p_{j+1}$ and for each $j \in P_2(i)$ with $j \le e$, we have $q_{j+1}s_i^\omega = q_{j+1} s_i = q_{j+1}$.
  
  By the definitions of $P_1(i)$ and $P_2(i)$, this demonstrates that $p_1 h(u) = p_{\ell+1}$ and $q_1 h(v) = q_{r+1}$ where $h \colon X^* \to S$ is the homomorphism defined by $h(x_i) = s_i^\omega$ for all $i \in [e]$ and $h(x_i) = s_i$ for all $i \in \{e+1,\dots,m\}$. By Line~2 of the algorithm, we obtain $h(u) \ne h(v)$, thereby concluding the proof.
\end{proof}

\section{Idempotent Membership}

\cite[Thm 4.10]{TJ:CI} proves the following problem is $\PSPACE$-complete:

\medskip
{\bf Idempotent Membership for Inverse Semigroups}
\begin{itemize}
\item Input: $a_1,\dots,a_k,b \in I_n$ where $bb = b$.
\item Problem: Is $b$ in the inverse semigroup generated by $a_1,\dots,a_k$?
\end{itemize}

Idempotent Membership for Partial Bijection Semigroups is similarly defined, asking whether $b$ is in the generated partial bijection semigroup. We now provide an independent proof that Idempotent Membership for Partial Bijection Semigroups is a $\PSPACE$-complete problem by reducing from the $\PSPACE$-complete Rectangle Tiling Problem. For this problem, we are given square tiles $T := \{T_1,\dots,T_k\}$, a set of colors $C := \{1,\dots,c\}$, and a positive integer $m$. Each side of each tile is a color from $C$. The problem is whether there is a way to arrange tiles from  $T$ into a grid of fixed width $m$ and some length $n$ such that adjacent edges have the same color and the edges along the border are all colored $1$ \cite{CHL:DTG}.

\medskip
{\bf Rectangle Tiling Problem}
\begin{itemize}
\item Input: $T = \{T_1,\dots,T_k\}$, $C = \{1,\dots,c\}$, and $m \in \mathbb{N}$ given in unary form.
\item Problem: Is there an $n \in \mathbb{N}$ and an $m \times n$ grid that is a proper tiling?
\end{itemize}

Formally, for an $m \times n$ grid of tiles, we can refer to the tile in the $i^{th}$ row and $j^{th}$ column as $T_{i,j}$ and think of it as a map $T_{i,j}: \{1,2,3,4\} \to C$ representing colors on its (1) north, (2) right, (3) south, and (4) left edges. The row values increase from north to south and the column values increase from left to right. We define a grid of tiles to be a {\bf proper tiling} if every outside edge has color 1 and adjacent edges have matching colors. In other words, a tiling is proper if the following conditions are true for every $i,j \geq 1$:
\begin{align*}
T_{1,j}(1) &= T_{i,n}(2) = T_{m,j}(3) = T_{i,1}(4) = 1,\\
T_{i,j}(2) &= T_{i,j+1}(4),\text{ and } T_{i,j}(3) = T_{i+1,j}(1).
\end{align*}
Placing tiles north-to-south and columns left-to-right, each tile can be thought of as a partial bijection that: (1) maps its north color to its south color, (2) maps its left color to its right color, and (3) preserves the left/right colors of other rows. We formally encode this intuition in the following proof.

\begin{proposition} \label{idmemthm}
  Idempotent Membership for Partial Bijection Semigroups is $\PSPACE$-complete.
\end{proposition}
\begin{proof}
Given tiles $T = \{T_1,\dots,T_k\}$, colors $C = \{1,\dots,c\}$, and a fixed height $m \in \mathbb{N}$, we will describe how each tile can be thought of as partial bijections in such a way that a proper tiling exists iff the generated partial bijection semigroup contains an idempotent that fixes white colors. The points acted on by the partial bijections semigroup will consist of $m$ distinct copies of $C$ for north-south colors and $m$ distinct copies of $C$ for left-right colors: $Q := \{(p,q,r):p \in \{h,v\}, q \in [m],r \in [c]\}$. The point $(v,i,r)$ represents a north-south color in the $i^{th}$ row. The point $(h,i,r)$ represents a left/right color in the $i^{th}$ row.

For each tile $T_j$, we define generators $\{a_{1,j},\dots,a_{m,j}\}$ corresponding to how $T_j$ maps colors when placed into a particular row. For $1 \leq i < m$, $a_{i,j}$ maps the north-south color $(v,i,T_j(1))$ to the north-south color $(v,i+1,T_j(3))$. $a_{m,j}$ maps $(v,m,T_j(1))$ to $(v,1,1)$ iff $T_j(3)=1$. $a_{i,j}$ is undefined on all other north-south colors. The first index of the north-south colors will force generators to be composed in the following form: $a_{1,\_} \cdots a_{m,\_}$ This corresponds to tiles being placed north-to-south in a column with matching north-south edges, ending with a tile that has a southern color of $1$. $a_{i,j}$ maps the left-right color $(h,i,T_j(4))$ to the left-right color $(h,i,T_j(2)$. $a_{i,j}$ is undefined on colors $(h,i,r)$ with $r \neq T_j(4)$ and it fixes left-right colors from every other row, $(h,p,r)$ with $p \neq i$.

We claim that there is a proper tiling iff $\langle a_{v,1,1},\dots,a_{m,k}\rangle$ contains an idempotent with domain $\{(v,1,1),(h,1,1),\dots,(h,m,1)\}$, corresponding to $1$ being the value of the first northern color and the value of the left-most colors of each row. Suppose the idempotent exists. To preserve the point $(v,1,1)$, the first generator must be of the form $a_{1,j}$ where $T_j(1) = 1$, meaning the first tile has a northern color of $1$. We previously noted that, in order to keep $(1,1)$ in the domain of the generated element, the first $m$ generators will force a corresponding column of tiles with matching north-south edges and a final southern color of $1$. Since $(v,m,T_j(1))a_{m,j}=(v,1,1)$ for every $T_j$ with a southern edge colored $1$, then $a_{1,\_} \cdots a_{m,\_}$ fixes $(v,1,1)$ and the southern edge of the $m^{th}$ tile is colored $1$. Inductively, the generated idempotent must be a composition of elements of the form $a_{1,\_} \cdots a_{m,\_}$, which corresponds to a tiling that: (1) has northern-most and southern-most edges colored $1$ and (2) matching adjacent northern and southern edges.

To preserve the color $(h,i,1)$, each of the first $m$ generators must correspond to tiles that have a left color of $1$, since the only point preserved by $a_{i,j}$ in row $i$ is the color $T_j(4)$. Note that each $a_{i,j}$ fixes points in every other row. Then the first $m$ generators, $a_{1,j_1} \cdots a_{m,j_m}$, will map each $(h,i,1)$ to $(h,i,T_{j_i}(2))$, the color on the right edge of the corresponding tile. Inductively, subsequent columns will each have left edges that match the previous column's right edges. The condition that each $(h,i,1)$ be mapped back to itself forces the final column to have right edges colored $1$.

For the converse, suppose a proper tiling exists and denote the tile in the the $(m,n)$ position as $T_{[m,n]}$. Then we claim the idempotent that fixes $\{(v,1,1),(h,1,1),\dots,(h,m,1)\}$ can be obtained by the following composition of generators:

\[\prod \limits_{p = 1} \limits^n \prod \limits_{i=1} \limits^m a_{i,[i,p]}.\]

We need to check that this generated element fixes $\{(v,1,1),(h,1,1),\dots,(h,m,1)\}$ and excludes all other colors from its domain. Because $T_{[1,1]}(1)=1$, then $(v,1,1)a_{1,[1,1]} = (v,2,T_{[1,1]}(3))$ and all other north-south colors will be excluded. Since $T_{[i,1]}(1)=T_{[i-1,j]}(3)$ for $1<i\leq m$, then $(v,1,1)$ will remain in the domain of the composition of the first $m$ generators. Since $T_{[m,1]}(3)=1$, then the first $m$ generators will map the north-south color $(v,1,1)$ back to itself. This argument is valid for each of the $n$ blocks of $m$ generators, so the whole composition will fix $(v,1,1)$.

Each left-right color $(h,i,1)$ can only be excluded by generators $a_{i,[i,p]}$, being fixed by all other generators. Since the first tile in each row has a left edge of $1$, the first $m$ generators will map $(h,i,1)$ to $(h,i,T_{[i,1]}(2))$ and left-right colors $(h,i,r)$ with $r \neq 1$ will be excluded. Inductively, because the left edges of each column match the right edges of the previous column, the point $(h,i,1)$ will remain in the domain of the composition. Because the final right edges are all colored $1$, the point $(h,i,1)$ will be mapped back to itself.
\end{proof}

Variations of the Rectangle Tiling Problem yield complete problems for other complexity classes. The Square Tiling Problem, in which a given value defines both the width and height of the grid, is known to be $\NP$-complete \cite{CHL:DTG}. The Domino Problem, which asks whether a given finite set of tiles can tile the quarter plane, is known to be undecidable \cite{BE:UD}. Consequently, understanding the connection between tiling problems and partial bijection semigroup problems yields a powerful tool for analyzing the complexity of partial bijection semigroup problems. For example, the following problem is certainly in $\NP$.

\medskip
{\bf Restricted Idempotent Membership for Inverse Semigroups}
\begin{itemize}
\item Input: $m \in \mathbb{N}$, $a_1,\dots,a_k,b \in I_n$ where $bb = b$.
\item Problem: Is $b = c_1 \cdots c_{m^2}$ for some $c_1,\dots,c_{m^2} \in \{a_1,\dots,a_k\}$?
\end{itemize}

That it is $\NP$-Complete follows by reducing from the Square Tiling Problem using a similar construction as described in the proof of Prop~\ref{idmemthm}.

\section{Acknowledgements}
The author would like to thank Alan Cain, Ant\'onio Malheiro, and Peter Mayr for their valuable comments and contributions.


\end{document}